\documentclass[11pt,twoside,a4paper]{amsart}
\usepackage{amsmath,amssymb,amsthm}
\usepackage{graphicx}
\usepackage[english]{babel}
\usepackage{hyperref}
\usepackage{enumerate}
\usepackage[margin=1.5in,head=13.6pt]{geometry}

\theoremstyle{plain}
\newtheorem{thm}{Theorem}[section]

\newtheorem{lem}[thm]{Lemma}

\newtheorem{defn}[thm]{Definition}

\newtheorem{prop}[thm]{Proposition}
\usepackage{color}

\newcommand{\V}{\mathcal{V}}

\DeclareMathOperator{\Aut}{Aut}
\DeclareMathOperator{\Tr}{Tr}

\begin{document}

\title{Switching equivalence of strongly regular polar graphs}

\author{G\'{a}bor P. Nagy}
\address{Bolyai Institute, University of Szeged, Aradi V\'{e}rtan\'{u}k tere 1, H-6720 Szeged, Hungary; and 
Institute of Mathematics, Budapest University of Technology and Economics, M\H{u}egyetem rkp. 3, H-1111 Budapest, Hungary.}
\email{nagyg@math.u-szeged.hu, nagy.gabor.peter@ttk.bme.hu}
\author{Valentino Smaldore}
\address{Dipartimento di Tecnica e Gestione dei Sistemi Industriali,
Universit\`{a} degli Studi di Padova, Stradella S. Nicola 3, 36100 Vicenza, Italy.}
\email{valentino.smaldore@unipd.it}

\maketitle
\begin{abstract}
We prove the switching equivalence of the strongly regular polar graphs $NO^\pm(4m,2)$, $NO^\mp(2m+1,4)$, and $\Gamma(O^\mp(4m,2))$ with an isolated vertex by giving an analytic description for them and their associated two-graphs.
%In this paper, we investigate the two-graphs associated with strongly regular polar graphs that belong to three infinite classes. Specifically, we examine the strongly regular graph $\Gamma(O^\pm(2m,2))$, which has vertices representing points of a nondegenerate hyperbolic or elliptic quadric $Q^\pm(2m-1,2)$ in the projective space $PG(2m-1,2)$. The set of vertices for $NO^\pm(2m,2)$ is the complement of $Q^\pm(2m-1,2)$. Additionally, we consider the vertices of the third graph $NO^\pm(2m+1,q)$, where $q$ is even, which correspond to hyperplanes of $PG(2m,q)$ that intersect the nondegenerate parabolic quadric in a nondegenerate hyperbolic or elliptic quadric.
%Our main result is the proof of switching equivalence for the strongly regular polar graphs $NO^\pm(4m,2)$, $NO^\mp(2m+1,4)$, and $\Gamma(O^\mp(4m,2))$ with an isolated vertex. We establish this by providing an analytic description for these graphs and their corresponding two-graphs.
\end{abstract}
%\tableofcontents

\section*{Introduction}
A \textit{strongly regular graph} with parameters $(v,k,\lambda,\mu)$ is a graph with $v$ vertices where each vertex is incident with $k$ edges, any two adjacent vertices have $\lambda$ common neighbors, and any two non-adjacent vertices have $\mu$ common neighbors. Two strongly regular graphs have the same parameters if and only if they are cospectral, see for instance \cite{brvan}. Strongly regular graphs were introduced by R. C. Bose in \cite{Bose} in 1963, and ever since then they have been intensively investigated. 

Many interesting and well-studied examples of strongly regular graphs belong to polar spaces. The \textit{tangent graph} of a polar space $\mathcal{P}$ embedded in a projective space is defined as the graph in which the vertices are non-isotropic points of the polarity that defines $\mathcal{P}$, and two vertices are adjacent if and only if they lie on the same tangent line to $\mathcal{P}$. In \cite{brvan} it is shown that in some cases tangent graphs are strongly regular. Here we focus on the graph $NO^{\pm}(2m,2)$, whose vertices are non-isotropic points w.r.t. a non-degenerate hyperbolic ($+1$ case) or elliptic ($-1$ case) quadric $Q^{\pm}(2m-1,2)$ in a projective space of odd dimension over the binary field. See \cite{brvan} for more details.

The graph $NO^{\pm}(2m+1,q)$ is the graph defined on the non-singular hyperplanes of one type with respect to a parabolic quadric $Q(2m,q)$ in a projective space of even dimension, where two vertices are adjacent if the intersection of the corresponding hyperplanes meets the quadric in a degenerate section. This construction belongs to an idea of Wilbrink, reported later in \cite{brVL}. This paper aims to show the switching equivalence between the graph $NO^{\pm}(4m,2)$ and the graph $NO^{\mp}(2m+1,4)$.

Graphs are called \textit{switching equivalent} if we can obtain one from the other by applying a series of \textit{switching} of edges and non-edges between subsets of vertices. We will be mainly interested in the \textit{Seidel switching}, introduced by Seidel in \cite{seidel}, related to the concept of two-graph. Seidel switching under certain conditions can modify the adjacency matrix of a given graph to produce cospectral graphs.

A two-graph is a set of vertices and an incidence relation for triples of vertices. A two-graph is regular if each pair of vertices is contained in a constant number of triples. See \cite{BH, Tay} for more information on two-graphs, while in \cite{table} the authors made a list of two-graphs of small order. Two graphs are switching equivalent if and only if they have the same associated two-graph.

The paper is organized as follows. In Section \ref{sec:prelim} we give some preliminaries on symplectic and orthogonal polarities in characteristic two, while in Section \ref{sec:srg-two-graphs} we focus on strongly regular polar graphs, Seidel switching, and two-graphs. In Section \ref{sec:analytic} we give useful analytic descriptions of the graphs $NO^{\pm}(2m,2)$, $NO^{\mp}(2m+1,q)$, $q=2^h$. In Section \ref{sec:main-two-graphs} we define the classes $\mathcal{X}^\pm_{2m}$ of two-graphs with doubly transitive automorphism group $Sp(2m,2)$, and prove the main theorem on the two-graphs of $NO^{\pm}(4m,2)$, $NO^{\mp}(2m+1,4)$. Finally, Appendix \ref{sec:appendix} reports some useful concepts and tools about symplectic groups and quadrics in characteristic two.

\section{Symplectic and orthogonal groups in characteristic two}\label{sec:prelim}

\subsection{Quadratic forms and alternating bilinear forms}
Consider the $2m$-dimensional vector space $V=\mathbb{F}_q^{2m}$, $q=2^{h}$, of row vectors. The map $\Theta:V\to \mathbb{F}_q$ is a \textit{quadratic form,} if $\Theta(\lambda u)=\lambda^2\Theta(u)$ for all $\lambda\in \mathbb{F}_q$, $u\in V$, and
\[f(u,v)=\Theta(u+v)+\Theta(u)+\Theta(v)\]
is a bilinear form on $V$. In this case, we say that $\Theta$ \textit{linearizes} to $f$. In even dimension, the quadratic form $\Theta$ is \textit{non-degenerate,} if its associated bilinear form is non-degenerate. We use the notation $u^\perp, U^\perp$ for the orthogonal complement of $u\in V$, or $U\subset V$. In characteristic two, quadratic forms always linearize to alternating bilinear forms: $f(u,u)=0$ for all $u\in V$. A bilinear form is \textit{symplectic} if it is non-degenerate and alternating.

Take the matrices
\[E=\Big(\begin{array}{cc}
    0 & I_{m} \\
    0 & 0
\end{array}\Big)\]
and \[F=E+E^{T}=\Big(\begin{array}{cc}
    0 & I_{m} \\
    I_{m} & 0
\end{array}\Big).\] 
We define the symplectic form $\langle .,. \rangle$ by
\begin{equation}\label{eq:ALT}
 \langle u,v\rangle=uFv^{T}. 
\end{equation}
Throughout this paper, we denote by $\Omega$ the set of quadratic forms $V\to \mathbb{F}_q$ that linearize to $\langle .,. \rangle$. Then 
\[\vartheta_{0}(u)=uEu^{T}\]
and 
\begin{equation} \label{eq:theta_a}
\vartheta_{a}(u)=\vartheta_{0}(u)+\langle a,u\rangle^{2}
\end{equation}
are both quadratic forms in $\Omega$. The map $\Theta+\vartheta_0$ is additive and $(\Theta+\vartheta_0)(\lambda u)=\lambda^2 (\Theta+\vartheta_0)(u)$. As elements in $\mathbb{F}_q$ have a unique square root, $(\Theta+\vartheta_0)^\frac{1}{2}$ is a well-defined linear map $V\to \mathbb{F}_q$. This implies $\Theta=\vartheta_a$ for some $a\in \mathbb{F}_q^{2m}$, and
\[\Omega=\{\vartheta_{a}\mid a\in\mathbb{F}_q^{2m}\}.\]

Let $\Theta:\mathbb{F}_q^{2m} \to \mathbb{F}_q$ be a non-degenerate quadratic form. Then $\Theta$ has Witt index $m-1$ or $m$. In the first case, $\Theta$ is \textit{elliptic,} and it is equivalent to
\[h(x_1,x_2)+x_3x_4+\cdots+x_{2m-1}x_{2m},\]
where $h(x,y)$ is an irreducible quadratic form in two variables. If the Witt index is $m$, then $\Theta$ is \textit{hyperbolic,} and it is equivalent to 
\[x_1x_2+x_3x_4+\cdots+x_{2m-1}x_{2m}.\]
Elliptic or hyperbolic type quadratic forms are also called type $-1$ or type $+1$, respectively. $\vartheta_0$ is of hyperbolic type.

In odd dimension, the behavior of quadratic and symplectic bilinear forms is slightly different. The quadratic form $\Theta:\mathbb{F}_q^{2m+1} \to \mathbb{F}_q$ has Witt index at most $m$. Quadratic forms of maximum Witt index are non-degenerate. Such forms are equivalent to
\[x_1^2+x_2x_3+\cdots+x_{2m}x_{2m+1}.\]
In this case, we say that $\Theta$ is \textit{parabolic,} or that $\Theta$ is of type $0$. 

\subsection{Symplectic and orthogonal transformations}
As before, $V=\mathbb{F}_q^{2m}$, $q=2^h$. The linear map $A:V\to V$ is a \textit{symplectic transformation} w.r.t. the symplectic form $\langle .,. \rangle$ provided
\[\langle uA,vA\rangle=\langle u,v\rangle\]
holds for all $u,v\in V$. As $\langle .,. \rangle$ is non-degenerate, a symplectic transformation is invertible; they form the group symplectic group $Sp(2m,q)$. The affine symplectic group $ASp(2m,q)$ consists of maps $u\mapsto uA+b$, where $A\in Sp(2m,q)$ and $b\in V$. Since $AFA^{T}=F$ we see that $Sp(2m,q)\leq SL(2m,q)$, i.e., all symplectic transformations have determinant $1$. 

For any linear transformation $A:V\to V$ and quadratic form $\Theta:V\to \mathbb{F}_q$, we define the map
\[\Theta^A:V\to \mathbb{F}_q, \qquad \Theta^A(u)=\Theta(uA^{-1}).\]
$\Theta^A$ is a quadratic form, and $(\Theta,A)\mapsto \Theta^A$ is a group action. If $A$ is symplectic w.r.t. $\langle .,. \rangle$, then $A$ preserves $\Omega$. In particular, $Sp(2m,q)$ acts on $\Omega$. 

Let $\Theta$ be a fixed quadratic form, and let $\varepsilon\in\{-1,0,1\}$ denote its type. The \textit{orthogonal group} (associated to $\Theta$) is defined as
\[O^\varepsilon(n,q) = \{A\in GL(n,q) \mid \Theta(uA)=\Theta(u) \text{ for all $u \in V$} \}.\]

Let $n=2m+1$ and $\Theta$ a parabolic quadratic form. If we linearize $\Theta$, then we obtain a degenerate alternating form with $1$-dimensional radical $V_0$. The factor space $V/V_0$ has a non-degenerate alternating form. This induces the group isomorphism
\[O(2m+1,q)\cong Sp(2m,q)\]
for finite fields of characteristic two, see \cite{grove}. 

The following three theorems summarize common knowledge on the structure of symplectic groups and their action on quadratic forms. Since it is hard to find concise proof for these properties in the literature, we give the basic ideas in the Appendix.

\begin{thm} \label{th:structure_sp_2m_q}
    Let $q$ be a power of two, and $\langle .,. \rangle$ a symplectic bilinear form on $\mathbb{F}_q^{2m}$. Let $\Omega$ be the set of quadratic forms $V\to \mathbb{F}_q$ which linearize to $\langle .,. \rangle$. 
    \begin{enumerate}[(i)]
        \item The symplectic group $Sp(2m,q)$ has two orbits 
        \begin{align*}
            \Omega^+ &= \{\vartheta_a \in \Omega \mid \Tr_{\mathbb{F}_q/\mathbb{F}_{2}}(\vartheta_0(a))=0\}, \\
            \Omega^- &= \{\vartheta_a \in \Omega \mid \Tr_{\mathbb{F}_q/\mathbb{F}_{2}}(\vartheta_0(a))=1\}
        \end{align*}
        on  $\Omega$.
        \item $|\Omega^+|=\frac{1}{2}(q^{2m}+q^{m})$ and $|\Omega^-|=\frac{1}{2}(q^{2m}-q^{m})$. 
        \item Let $R^+$ (resp. $R^-$) be the stabilizer subgroup of $\Theta\in \Omega^+$ (resp. $\Theta\in \Omega^-$) in $Sp(2m,q)$. $R^+$ (resp. $R^-$) is isomorphic to $O^+(2m,q)$ (resp. $O^-(2m,q)$), and its action on $\Omega$ is equivalent to the action of $O^+(2m,q)$ (resp. $O^-(2m,q)$) on $V$.
    \end{enumerate}
\end{thm}

The case $q=2$ has some further important properties. 

\begin{thm} \label{th:action_sp_2m_2}
    $Sp(2m,2)$ acts $2$-transitively on $\Omega^+$ and on $\Omega^-$. $R^+$ (resp. $R^-$) acts transitively on $\Omega^-$ (resp. $\Omega^+$). 
\end{thm}
        
The translations $\tau_d:u\mapsto u+d$, $d\in V=\mathbb{F}_2^{2m}$, form an elementary abelian group $N$, acting regularly on $V$. Any linear transformation $A$ normalizes $N$. The \textit{affine symplectic group} $ASp(2m,q)$ is the closure group of $N$ and $Sp(2m,q)$. 

\begin{thm} \label{th:symple-complements}
In the affine symplectic group $ASp(2m,q)$, the regular normal subgroup $N$ has two complements $H_1,H_2\cong Sp(2m,q)$ such that the orbits of $H_1$ have length $1$ and $q^{2m}-1$ and the orbits of $H_2$ have length $\frac{1}{2}(q^{2m}+q^m)$, $\frac{1}{2}(q^{2m}-q^m)$. The actions of $H_2$ on the orbits are equivalent to the actions of $Sp(2m,q)$ on $\Omega^+$ and $\Omega^-$. 
\end{thm}

\section{Strongly regular graphs, Seidel switching, two-graphs}\label{sec:srg-two-graphs}

\subsection{Preliminaries on projective quadrics}

We give a short overview of the notions related to the geometry of quadrics in finite projective spaces; for more details, see \cite{Hirschfeld3}. We deal only with the case of characteristic $2$. 

Let $V(n+1,q)$ be a vector space of dimension $n+1$ over the finite field $F_q$. The \textit{projective space} $PG(n,q)$, or $PG(V)$, is the set of all $1$-dimensional subspaces of $V$, called \textit{projective points}. A non-zero vector $a\in V$ determines the projective point $P(a)=\langle a \rangle$. Let $W$ be an $(h+1)$-dimensional subspace of $V$. The $1$-dimensional subspaces of $W$ form a projective subspace $[W]=PG(h,q)$; we say that $[W]$ has projective dimension $h$. 

Let $\Theta:V\to \mathbb{F}_q$ be a quadratic form on $V$. Let $\V(\Theta)$ denote the set of projective points $P(a)$ such that $\Theta(a)=0$. This is a well-defined quadratic surface since $\Theta$ is homogeneous of degree $2$. The quadric $Q=\V(\Theta)$ is non-singular, if $\Theta$ is non-degenerate. We say that $Q$ is \textit{elliptic, parabolic} or \textit{hyperbolic,} if $\Theta$ is a non-degenerate quadratic form of the appropriate type. We also denote the elliptic, parabolic, and hyperbolic quadrics of $PG(n,q)$ by $Q^-(n,q)$, $Q(n,q)$, and $Q^+(n,q)$, respectively. 

Let $Q$ be a non-singular quadric in $PG(n,q)$ and $S$ a projective subspace of projective dimension $h>1$. $S\cap Q$ is a quadric in $S$. We say that $S$ is \textit{tangent} or \textit{secant} to $Q$ if $S\cap Q$ is singular or non-singular in $S$, respectively. A line $\ell$ is tangent to $Q$ if and only if $|\ell\cap Q|=1$ or $\ell \subseteq Q$. 

Let $Q=\V(\Theta)$ be a non-singular quadric in $PG(n,q)$, $n$ odd. Then the underlying linear space $V$ has even dimension $n+1$. The alternating form $\langle .,. \rangle$ associated with $\Theta$ is non-degenerate, and it determines a symplectic polarity of $V$. The polar $S^\perp$ of a projective subspace $S$ of projective dimension $h$ has dimension $n-h-1$. In particular, the polar of a point is a hyperplane; $P\in Q$ if and only if $P^\perp$ is a tangent hyperplane to $Q$. 

Let $Q=\V(\Theta)$ be a non-singular quadric in $PG(n,q)$, $n$ even. Then, $\langle .,. \rangle$ is singular with a $1$-dimensional kernel $\langle u \rangle$, such that $\Theta(u)\neq 0$. The projective point $N=P(u)$ is called the \textit{nucleus} of $Q$. The hyperplane is tangent to $Q$ if and only if it contains $N$. If the hyperplane $H$ does not contain $N$, then $H\cap Q$ is a non-singular quadric in $PG(n-1,q)$, of type $+1$ (hyperbolic) or of type $-1$ (elliptic).

\subsection{Strongly regular polar graphs}
%A strongly regular graph $srg(v,k,\lambda,\mu)$ is a $k$-regular simple graph on $v$ vertices such that any two adjacent vertices have $\lambda$ common neighbors, and any two non-adjacent vertices have $\mu$ common neighbors. Two strongly regular graphs have the same parameters if and only if they are cospectral. 

In this subsection, we recall the definition of two infinite classes of strongly regular graphs, see \cite[Section 3.1]{brvan} for the proofs of the Lemmas \ref{lm:no-2m-2} and \ref{lm:no-2m+1-q} and further detail. 

Consider a non-singular quadric $Q^{\pm}(2m-1,2)$ in $PG(2m-1,2)$, of type $\pm1$. Then $|Q^{\pm}(2m-1,2)|=2^{2m-1}\pm2^{m-1}-1$, and $|PG(2m-1,2)\setminus Q^{\pm}(2m-1,2)|=2^{2m-1}\mp2^{m-1}$.
\begin{defn}
    The graph $NO^{\pm}(2m,2)$ is the graph whose vertex set is $PG(2m-1,2)\setminus Q^{\pm}(2m-1,2)$, and two vertices are adjacent if and only if the two corresponding points lie on a line tangent to the quadric $Q^{\pm}(2m-1,2)$.
\end{defn}
\begin{lem}\cite[Section 3.1.2]{brvan} \label{lm:no-2m-2}
 $NO^{\pm}(2m,2)$ is a strongly regular graph with parameters 
  \begin{align*}
 v &= 2^{2m-1}\mp2^{m-1}, \\
 k &= 2^{2m-2}-1, \\
 \lambda &= 2^{2m-3}-2, \\
 \mu &= 2^{2m-3}\pm2^{m-2}.
 \end{align*}
The automorphism group of $NO^{\pm}(2m,2)$ contains the orthogonal group $O^{\pm}(2m,2)$.
\end{lem}

Now consider a non-singular parabolic quadric $Q(2m,q)$ in $PG(2m,q)$, $q$ even. Let $\mathcal{H}^{\pm}$ be the set of hyperplanes not containing the nucleus $N$ and intersecting the quadric in a non-singular quadric of type $\pm1$.
\begin{defn}
    The graph $NO^{\pm}(2m+1,q)$ is the graph whose vertex set is $\mathcal{H}^{\pm}$, and two vertices $H_1$ and $H_2$ are adjacent if and only if $H_1\cap H_2\cap Q(2m,q)$ is degenerate.
\end{defn}
\begin{lem}\cite[Section 3.1.4]{brvan} \label{lm:no-2m+1-q}
 $NO^{\pm}(2m+1,q)$ is a strongly regular graph with parameters 
 \begin{align*}
 v &= \frac{1}{2}q^m(q^m\pm1),\\
 k &= (q^{m-1}\pm1)(q^m\mp1),\\
 \lambda &= 2(q^{2m-2}-1)\pm q^{m-1}(q-1),\\
 \mu &= 2q^{m-1}(q^{m-1}\pm1).    
 \end{align*} 
The automorphism group of $NO^{\pm}(2m+1,q)$ contains $O(2m+1,q)\cong Sp(2m,q)$.
\end{lem}

\subsection{Two-graphs and switchings}
We start this subsection by defining the concepts of \textit{two-graph}, and of two-graph associated with a given graph.
\begin{defn}[Two-graph]
 A two-graph is a pair $(X, T)$, where $T$ is a set of unordered triples of a vertex set $X$, such that every (unordered) quadruple from $X$ contains an even number of triples from $T$. In a \textit{regular two-graph}, each pair of vertices is in a constant number of triples.
\end{defn}
\begin{defn}[Associated two-graph]
 Given a graph $G = (V,E)$, the set of triples $T$ of the vertex set $V$ whose induced subgraph has an odd number of edges forms a two-graph on the set $V$. The two-graph $Tau(G)=(V,T)$ is the \textit{associated two-graph} of $G$.
\end{defn}

Let $\mathcal T=(V,T)$ be a two-graph. The \textit{descendant} of $\mathcal T$ w.r.t. the vertex $w\in V$ is the graph $\mathcal T_w=(V,E)$ where $E$ consists of pairs $(u,v)$ such that $\{u,v,w\}\in T$. By definition, $w$ is an isolated vertex of the descendant. It is not hard to see that $Tau(\mathcal T_w)=\mathcal T$, which shows that every two-graph can be represented as an associated two-graph of a graph.

\begin{defn}[Seidel switching \cite{seidel}]
Given a graph $G=(V,E)$ and a subset $Y$ of the vertex set, the operation of switching $G$ with respect to $Y$ consists of replacing all edges from $Y$ to its complement by nonedges, and all nonedges by edges, while leaving the edges within $Y$ or outside $Y$ unchanged.
\end{defn}
Switching defines an equivalence relation on the set of all graphs with a given $n$-element vertex set, each equivalence class containing $2^{n-1}$ graphs (since switching with respect to a set and its complement are the same thing, and switching with respect to two sets is the same as switching with respect to their symmetric difference). The equivalence classes are called \textit{switching classes}.

The following proposition is well-known; we give the proof for the sake of completeness.

\begin{prop}\cite{seidel} \label{pr:switching-equiv-same}
Two graphs are switching equivalent if and only if they have the same associated two-graph. 
\end{prop}
\begin{proof}
Consider a subset $Y\subseteq V$ as above. W.l.o.g. we say that each triple $t=\{u,v,z\}$ either has all vertices in $Y$, or two vertices in $Y$ and one in $V\setminus Y$. In the former case, the adjacencies in $t$ are fixed since the adjacencies in a component $G_i$ are not involved in the switching. In the latter case, we easily check that the switching preserves the parity of the number of edges between $u$, $v$, and $z$ (we have 1, 2, or 0 edges through $Y$ and $V\setminus Y$, and the Seidel switching deletes the edge and adds another one, deletes both edges, or adds two edges, respectively). 

Conversely, assume that $G_1=(V,E_1)$ and $G_2=(V,E_2)$ have the same associated two-graph $\mathcal T$. Let $F$ be the symmetric difference of $E_1$ and $E_2$. It is straightforward to check that the graph $H=(V,F)$ has the associated two-graph empty. This implies that the non-adjacency relation is an equivalence relation on $V$, with at most $2$ equivalence classes. Choosing $Y$ to be one of the equivalence classes, we find that $G_2$ is the switching of $G_1$ w.r.t. $Y$. 
\end{proof}

\section{Analytic description of strongly regular polar graphs} \label{sec:analytic}

Throughout this section, notation will follow Sections \ref{sec:prelim} and \ref{sec:srg-two-graphs}.

\begin{prop}[Analytic description of $NO^\pm(2m,2)$] \label{pr:anal1}
Let $\Theta$ be a non-degener\-ate quadratic form on $\mathbb{F}_2^{2m}$ of type $\pm1$. Let $\langle .,. \rangle$ be the symplectic form associated to $\Theta$. Define the graph $(V,E)$ with 
\begin{align*}
    V&=\{a\in \mathbb{F}_2^{2m} \mid \Theta(a)=1\}, \\
    E&=\{(a,b)\in V\times V \mid \langle a,b \rangle =0 \text{ and } a\neq b\}.
\end{align*} 
Then $(V,E)\cong NO^{\pm}(2m,2)$.
\end{prop}
\begin{proof}
    Let $Q^{\pm}(2m-1,2)=\V(\Theta)$ correspond to the quadratic form $\Theta$. Projective points $PG(2m-1,2)\setminus Q^\pm(2m-1,2)$ have coordinate vectors $a\in \mathbb{F}_2^{2m}$ with $\Theta(a)=1$. Hence, the vertices of $NO^{\pm}(2m,2)$ are in $1$-to-$1$ correspondence with our set $V$. For $a,b\in V$, the projective line $P(a)P(b)$ is consists of the three points $\{P(a),P(b),P(a+b)\}$ and 
    \begin{align*}
        a\sim b &\Leftrightarrow P(a+b)\in Q^{\pm}(2m-1,2) \\
        &\Leftrightarrow\Theta(a+b)=0 \\
        & \Leftrightarrow\Theta(a)+\Theta(b)+\langle a,b\rangle=0 \\
        & \Leftrightarrow\langle a,b\rangle=0. \qedhere
    \end{align*}
\end{proof}

Using quadratic forms, we aim to give an equivalent (analytic) description of $NO^\pm(2m+1,q)$ using quadratic forms. We say that two non-singular quadrics $Q_1,Q_2$ of $PG(2m-1,q)$ are \textit{tangent} if there is a hyperplane $H$ such that $H$ is tangent to both $Q_1,Q_2$, and $Q_1\cap H = Q_2 \cap H$. For the non-singular quadrics $Q_1=\V(\vartheta_a)$, $Q_2=\V(\vartheta_b)$, being tangent has a simple analytic characterization.
\begin{lem} \label{lm:tangent}
    For $\vartheta_a,\vartheta_b \in \Omega$, the following are equivalent:
    \begin{enumerate}[(i)]
        \item The non-singular quadrics $\V(\vartheta_a)$, $\V(\vartheta_b)$ are tangent.
        \item $\V(\vartheta_a) \cap \V(\vartheta_b)$ is a degenerate quadric.
        \item $P(a+b) \in \V(\vartheta_a) \cap \V(\vartheta_b)$.
        \item $\vartheta_a(a+b)=0$.
    \end{enumerate}
\end{lem}
\begin{proof}
    By the definition \eqref{eq:theta_a}, we have
    \begin{align} \label{eq:tangent-quadrics}
        \V(\vartheta_a) \cap \V(\vartheta_b) = \V(\vartheta_a) \cap H = \V(\vartheta_b) \cap H,
    \end{align}
    where $H=(a+b)^\perp$ is the hyperplane polar to $a+b$. Hence, $\V(\vartheta_a) \cap \V(\vartheta_b)$ is degenerate if and only if $H$ is tangent to both $\V(\vartheta_a)$ and $\V(\vartheta_b)$. This holds if and only if $H^\perp=P(a+b)$ is a point of both of them. This shows (ii)$\Leftrightarrow$(iii). The equivalence (iii)$\Leftrightarrow$(iv) is obvious. (ii), (iii) and equation \eqref{eq:tangent-quadrics} imply (i). 
    
    Assume (i): Let $K$ be a hyperplane which is tangent to both $\V(\vartheta_a)$, $\V(\vartheta_b)$, and $\V(\vartheta_a)\cap K = \V(\vartheta_b) \cap K$. Then $\V(\vartheta_a) \cap \V(\vartheta_b)\subseteq K$. However, $\V(\vartheta_a) \cap \V(\vartheta_b)$ spans $H$ by \eqref{eq:tangent-quadrics}, therefore $H=K$. As $H$ is tangent to $\V(\vartheta_a)$, $\V(\vartheta_b)$, we have $P(a+b)=H^\perp \in \V(\vartheta_a) \cap \V(\vartheta_b)$, and (iii) holds. 
\end{proof}

%Being tangent implies the type of quadratic forms.
\begin{lem}
    \begin{enumerate}[(i)]
        \item If $\V(\vartheta_a)$ and $\V(\vartheta_b)$ are tangent, then $\vartheta_a,\vartheta_b$ have the same type.
        \item If $q=2$ and $\vartheta_a$ and $\vartheta_b$ have the same type, then $\V(\vartheta_a)$ and $\V(\vartheta_b)$ are tangent.
    \end{enumerate}
\end{lem}
\begin{proof} 
    Since all $\vartheta_a$ linearize to $\langle .,. \rangle$, by definition we have
    \begin{align*}
        \vartheta_a(a+b)&= \vartheta_a(a)+\vartheta_a(b)+\langle a,b\rangle\\
        &=\vartheta_0(a)+\vartheta_0(b)+\langle a,b\rangle+\langle a,b\rangle^2.
    \end{align*}
    As $\Tr_{\mathbb{F}_q/\mathbb{F}_{2}}(\langle a,b\rangle) = \Tr_{\mathbb{F}_q/\mathbb{F}_{2}}(\langle a,b\rangle^2)$, the implications
    \begin{align*}
        \text{$\V(\vartheta_a)$, $\V(\vartheta_b)$ are tangent} &\Longleftrightarrow 
        \vartheta_a(a+b)=0 \\
        &\Longrightarrow \Tr_{\mathbb{F}_q/\mathbb{F}_{2}}(\vartheta_a(a+b))=0 \\
        &\Longleftrightarrow \Tr_{\mathbb{F}_q/\mathbb{F}_{2}}(\vartheta_0(a)) = \Tr_{\mathbb{F}_q/\mathbb{F}_{2}}(\vartheta_0(b)) \\
        &\Longleftrightarrow \text{$\vartheta_a$, $\vartheta_b$ have the same type}
    \end{align*}
    hold by Lemma \ref{lm:tangent} and Theorem \ref{th:structure_sp_2m_q}(i). If $q=2$, then the right implication above becomes equivalence, as well. 
\end{proof}

In the following lemma, we make a first attempt to interpret the graphs $NO^{\pm}(2m+1,q)$ analytically.

\begin{lem} \label{lm:anal2-prep}
Define the graph $(W^{\pm},F)$ with 
\begin{align*}
    W^+ & =\{a\in \mathbb{F}_q^{2m} \mid \Tr_{\mathbb{F}_q/\mathbb{F}_{2}}(\vartheta_0(a))=0\} \\
    W^- & =\{a\in \mathbb{F}_q^{2m} \mid \Tr_{\mathbb{F}_q/\mathbb{F}_{2}}(\vartheta_0(a))=1\} \\
    F & =\{(a,b)\in W^\pm\times W^\pm \mid \vartheta_a(a+b)=0 \text{ and } a\neq b\}.
\end{align*}
Then $(W^\pm,F) \cong NO^{\pm}(2m+1,q)$.
\end{lem}
\begin{proof}
Firstly see that $W^{\pm}$ and $\Omega^{\pm}$ are in 1-to-1 correspondence by associating $a\in W^{\pm}$ to $\vartheta_{a}\in\Omega^{\pm}$. Fix a hyperplane $H_0$ in $PG(2m,q)$ not through the nucleus $N=P(u)$ of the parabolic quadric $Q(2m,q)$. Any line through $N$ intersects both $Q$ and $H_0$ in unique points. In other words, the central projection $\pi$ from $N$ to $H_0$ is a bijection between $Q$ and $H_0$. Moreover, $\pi$ preserves the symplectic form, since $\langle u+x,u+y \rangle=\langle x,y \rangle$ holds for all $x,y\in \mathbb{F}_q^{2m+1}$

As $H_0$ is not tangent to $Q$, the restriction of $\langle .,. \rangle$ to $H_0$ is non-singular. Let $\Omega^{\pm}$ be the set of quadratic forms on $H_0$, linearizing to $\langle .,. \rangle|_{H_0}$. Moreover, since $\pi$ preserves the symplectic form, $\pi$ induces a bijection between $\mathcal{H}^{\pm}$ and $\Omega^{\pm}$. Indeed, if a certain hyperplane $H$ does not contain $N$, then $Q\cap H$ is non-singular in $H$ and $Q'=\pi(Q\cap H)$ is non-singular in $H_0$. There is a quadratic form $\Theta$ such that $Q'=\V(\Theta)$, and $\Theta$ linearizes to a multiple of $\langle .,. \rangle$ as $\pi$ preserves the symplectic structure. %By multiplying with an appropriate scalar, we can assume $\Theta \in \Omega$. 
In particular, $\Theta=\vartheta_a$ for some $a\in \mathbb{F}_q^{2m}$.
    
Fix $H_1,H_2 \in\mathcal{H}^{\pm}$ with $\pi(Q\cap H_1)=\V(\vartheta_a)$, $\pi(Q\cap H_2)=\V(\vartheta_b)$, $a,b \in \mathbb{F}_q^{2m}$, $\vartheta_a,\vartheta_b \in \Omega^{\pm}$. Since $Q\cap H_1$ is non-degenerate, 
\begin{align*}
H_1\sim H_2 &\Longleftrightarrow Q\cap H_1\cap H_2 \mbox{ degenerate} \\
&\Longleftrightarrow \pi(Q\cap H_1) \cap \pi(Q\cap H_2) \mbox{ degenerate} \\
& \Longleftrightarrow \V(\vartheta_a) \cap \V(\vartheta_b)\mbox{ degenerate} \\
%& \Longleftrightarrow \V(\vartheta_a), \V(\vartheta_b)\mbox{ tangent} \\
& \Longleftrightarrow \vartheta_a(a+b)=0,
\end{align*}
where the last implication follows from Lemma \ref{lm:tangent}.
\end{proof}

We can now analytically describe the graphs $NO^\pm(2m+1,q)$.

\begin{prop}[Analytic description of $NO^{\pm}(2m+1,q)$] \label{pr:anal2}
Let $\Theta$ be a non-degenerate quadratic form on $\mathbb{F}_q^{2m}$ of type $\mp1$. Let $\langle .,. \rangle$ be the symplectic form associated to $\Theta$. Define the graph $(V,E)$ with 
\begin{align*}
    V & =\{a\in \mathbb{F}_q^{2m} \mid \Tr_{\mathbb{F}_q/\mathbb{F}_{2}}(\Theta(a))=1\} \\
    E & =\{(a,b)\in V\times V \mid \text{$a\neq b$ and $\Theta(a+b)=\langle a,b \rangle^2$}\}.
\end{align*}
Then $(V,E) \cong NO^{\pm}(2m+1,q)$.
\end{prop}
\begin{proof}
    Let $\varepsilon=\pm 1$ be the sign of the graph; the quadratic form $\Theta$ has type $-\varepsilon$. Define $\lambda(\varepsilon)\in \mathbb{F}_2$ such that $\varepsilon=(-1)^{\lambda(\varepsilon)}$. Write $\Theta=\vartheta_d$ for some $d\in \mathbb{F}_q^{2m}$. Since $\vartheta_0$ has type $+1$, we have
    \[\Tr_{\mathbb{F}_q/\mathbb{F}_{2}}(\vartheta_0(d))=1+\lambda(\varepsilon)\]
    by Theorem \ref{th:structure_sp_2m_q}(i). This implies
    \begin{align*}
        \Tr_{\mathbb{F}_q/\mathbb{F}_{2}}(\vartheta_d(x)) &= \Tr_{\mathbb{F}_q/\mathbb{F}_{2}}(\vartheta_0(x) + \langle x,d \rangle^2) \\
        & =\Tr_{\mathbb{F}_q/\mathbb{F}_{2}}(\vartheta_0(x) + \langle x,d \rangle) \\
        & =\Tr_{\mathbb{F}_q/\mathbb{F}_{2}}(\vartheta_0(x+d) + \vartheta_0(d)) \\
        & =\Tr_{\mathbb{F}_q/\mathbb{F}_{2}}(\vartheta_0(x+d)) + 1+\lambda(\varepsilon). \\
    \end{align*}
    In other words, for all $x\in \mathbb{F}_q^{2m}$
    \[\Tr_{\mathbb{F}_q/\mathbb{F}_{2}}(\vartheta_d(x)) = 1 \Longleftrightarrow \Tr_{\mathbb{F}_q/\mathbb{F}_{2}}(\vartheta_0(x+d)) = \lambda(\varepsilon).\]
    This means that $W^\pm = V+d$, and the translation $\tau_d:x\mapsto x+d$ induces a bijection 
    \[V \longleftrightarrow W^\pm.\]
    Here, $W^\pm = \{a\in \mathbb{F}_q^{2m} \mid \Tr_{\mathbb{F}_q/\mathbb{F}_{2}}(\vartheta_0(a))=\lambda(\varepsilon)\}$ is defined as in Lemma \ref{lm:anal2-prep}. For the pair $(a,b) \in V\times V$, $a\neq b$,  we have
    \begin{align*}
        (a,b) \in E &\Longleftrightarrow \vartheta_d(a+b)=\langle a,b \rangle^2 \\
        &\Longleftrightarrow \vartheta_d(a+b)=\langle a+b,a \rangle^2 \\
        &\Longleftrightarrow \vartheta_{a+d}(a+b)+\vartheta_{a}(a+b)+\vartheta_{0}(a+b)+\langle a+b,a \rangle^2=0 \\
        &\Longleftrightarrow \vartheta_{a+d}(a+b)=0\\
        &\Longleftrightarrow (a+d,b+d) \in F.
    \end{align*}
    Again, $F$ is the set of edges of the graph in Lemma \ref{lm:anal2-prep}. What we see is that $\tau_d$ induces an isomorphism between the graphs $(V,E)$ and $(W^\pm,F)$. 
\end{proof}

\section{Two-graphs of strongly regular polar graphs}\label{sec:main-two-graphs}

\subsection{The symplectic two-graphs}
Define the graph $\Sigma_{2m}$ with vertex set $\mathbb{F}_{2}^{2m}$ and adjacency relation $\langle a,b\rangle=0$ for vectors $a\neq b$. This graph is not regular, since the origin is adjacent to all nonzero vectors. The automorphism group of $\Sigma_{2m}$ contains the linear symplectic group $Sp(2m,2)$. 

For non-adjacent vertices $a,b$ of $\Sigma_{2m}$, we have $\langle a,b \rangle =1$. The triple $\{a,b,c\}$ has an odd number of edges if and only if 
\begin{align} \label{eq:sympl-twograph}
    \langle a,b\rangle+\langle a,c\rangle+\langle b,c\rangle=0.
\end{align}
This means that the two-graph of $\Sigma_{2m}$ is the set of triples on $\mathbb{F}_{2}^{2m}$ that satisfy \eqref{eq:sympl-twograph}. We denote it by $\mathcal{T}_{2m}$. As
\begin{align*} 
    \langle a,b\rangle+\langle a,c\rangle+\langle b,c\rangle=\langle a+d,b+d\rangle+\langle a+d,c+d\rangle+\langle b+d,c+d\rangle,
\end{align*}
the translation $\tau_d:x\mapsto x+d$ stabilizes $\mathcal{T}_{2m}$ for all $d$. This implies that $ASp(2m,2)$ is a subgroup of $\Aut(\mathcal{T}_{2m})$. In particular, $\Aut(\mathcal{T}_{2m})$ is $2$-transitive and $\mathcal{T}_{2m}$ is a regular two-graph. 

By Theorem \ref{th:symple-complements}, there is a subgroup $H_2\cong Sp(2m,2)$ in $ASp(2m,2)$ with orbits $X^+$, $X^-$ of size $2^{2m-1}-2^{m-1}$ and $2^{2m-1}+2^{m-1}$. Moreover, $H_2$ acts $2$-transitively on $X^+,X^-$. We denote by $\mathcal{X}^\pm_{2m}$ the two-graph whose set of triples $\{a,b,c\}$ satisfies \eqref{eq:sympl-twograph}. Then $(X^+, \mathcal{X}^+_{2m})$, $(X^-, \mathcal{X}^-_{2m})$ are sub-two-graphs of $\mathcal{T}_{2m}$. As their automorphism group is $2$-transitive, $\mathcal{X}^+_{2m}$ and $\mathcal{X}^-_{2m}$ are regular two-graphs. 

The above definition is rather sloppy since the subgroup $H_2$ is defined only up to conjugacy in $ASp(2m,2)$, and the sets $X^+, X^-$ are defined up to translations. One possibility is to choose the sets
\begin{align*}
    X^+ &= \{a\in \mathbb{F}_{2}^{2m} \mid \vartheta_0(a)=1 \},\\
    X^- &= \{a\in \mathbb{F}_{2}^{2m} \mid \vartheta_0(a)=0 \}
\end{align*}
according to Theorem \ref{th:structure_sp_2m_q}. In the following proposition, we opt for a representation that is in agreement with the analytic descriptions of our strongly regular graphs (Propositions \ref{pr:anal1} and \ref{pr:anal2}). 

\begin{prop}\label{t-g}
    Let $\Theta$ be a non-degenerate quadratic form on $\mathbb{F}_2^{2m}$ of type $\pm1$. Let $\langle .,. \rangle$ be the symplectic form associated to $\Theta$. Define the set of vectors
    \[X=\{a\in \mathbb{F}_{2}^{2m} \mid \Theta(a)=1\}\]
    and the set of triples
    \[T=\{\{a,b,c\} \mid a\neq b\neq c\neq a, \; \langle a,b\rangle+\langle a,c\rangle+\langle b,c\rangle=0\}.\]
    Then the following hold:
    \begin{enumerate}[(i)]
        \item $\mathcal{X}^\pm_{2m}=(X,T)$ is a regular two-graph of degree $2^{2m-2}\mp 2^{m-1}-2$. 
        \item The automorphism group of $\mathcal{X}^\pm_{2m}$ is isomorphic to $Sp(2m,2)$ in its $2$-transitive action on $2^{2m-1}\mp 2^{m-1}$ points. 
    \end{enumerate}
\end{prop}
\begin{proof}
    We have seen in the proof of Proposition \ref{pr:anal2}, that for some $d\in \mathbb{F}_{2}^{2m}$, the translation $\tau_d$ induces a bijection between the sets $W^{\pm}$ and $V$, represented here by $X^\pm$ and $X$. Translations are contained in $ASp(2m,2)$, and preserve the relation \eqref{eq:sympl-twograph}, hence (ii) and the regularity holds. It only remains to compute the degree. Let $a,b \in \mathbb{F}_{2}^{2m}$ be different vectors such that their projective line $\ell$ is exterior to the projective quadric $\Theta=0$. Since the third point of $\ell$ is $a+b$, we have
    \[\Theta(a)=\Theta(b)=\Theta(a+b)=1.\]
    Hence $\langle a,b \rangle =1$. Fix any point $x\in \mathbb{F}_{2}^{2m}\setminus \{a,b\}$ with $\Theta(x)=1$. 
    \begin{align*}
        \{a,b,x\} \in T & \Longleftrightarrow \langle a,x\rangle+\langle b,x\rangle+\langle a,b\rangle= 0\\
        & \Longleftrightarrow \langle a+b,x \rangle =1 \\
        & \Longleftrightarrow \text{$a+b$ and $x$ are non-adjacent in $NO^\pm(2m,2)$.}
    \end{align*}
    By Lemma \ref{lm:no-2m-2}, the parameters of $NO^\pm(2m,2)$ are $v=2^{2m-1}\mp 2^{m-1}$ and $k=2^{2m-2}-1$. Therefore, the degree of the two-graph is
    \begin{align*}
        |\{x\in \mathbb{F}_{2}^{2m} \mid \{a,b,x\} \in T \}| &= v-k-1-2 \\
        &= (2^{2m-1}\mp 2^{m-1})-(2^{2m-2}-1)-3 \\
        &= 2^{2m-2}\mp 2^{m-1}-2. \qedhere
    \end{align*}
\end{proof}

\subsection{The descendants of symplectic two-graphs}
In this subsection, we introduce another structure of polar graph related to quadrics in odd projective dimension, over the binary field, see \cite[Sections 2.6.2 and 2.6.3]{brvan}.

\begin{defn}[Orthogonal graphs]
Let $q$ be a prime power. The graph $\Gamma(O^{\pm}(2m,q))$ is the graph whose vertex set is $Q^{\pm}(2m-1,q)$ and in which two vertices are adjacent if and only if the two corresponding points lie on a totally isotropic line of the quadric $Q^{\pm}(2m-1,q)$.
\end{defn}
\begin{lem} \cite[Theorem 2.2.12]{brvan}\label{lm:o-2m-q}
 $\Gamma(O^{\pm}(2m,q))$ is a strongly regular graph with parameters 
  \begin{align*}
 v &= \frac{(q^{m}\mp1)(q^{m-1}\pm1)}{q-1}, \\
 k &= \frac{q(q^{m-1}\mp1)(q^{m-2}\pm1)}{q-1}, \\
 \lambda &= \frac{q^2(q^{m-2}\mp1)(q^{m-3}\pm1)}{q-1}+q-1, \\
 \mu &= \frac{(q^{m-1}\mp1)(q^{m-2}\pm1)}{q-1}.
 \end{align*}
The automorphism group of $\Gamma(O^{\pm}(2m,q))$ contains the orthogonal group $O^{\pm}(2m,q)$.
%The automorphism group of $\Gamma(O^{\pm}(2m,2))$ is $P\Gamma O^{\pm}(2m,2)$, except from the case $Aut(\Gamma(O^+(4,2)))\cong S_{3}\wr C_2$.
\end{lem}

\begin{prop}[Analytic description of $\Gamma(O^\pm(2m,2))$] \label{pr:anal_d}
Let $\Theta$ be a non-degenerate quadratic form on $\mathbb{F}_2^{2m}$ of type $\pm1$. Let $\langle .,. \rangle$ be the symplectic form associated to $\Theta$. Define the graph $(V,E)$ with 
\begin{align*}
    V&=\{a\in \mathbb{F}_2^{2m}\setminus\{0\} \mid \Theta(a)=0\} \\
    E&=\{(a,b)\in V\times V \mid \langle a,b \rangle =0 \text{ and } a\neq b\}.
\end{align*} 
Then $(V,E)\cong \Gamma(O^{\pm}(2m,2))$.
\end{prop}

\begin{proof}
    The proof works as in Proposition \ref{pr:anal1}.
\end{proof}
Recall that the descendant of a two-graph $\mathcal T=(X,T)$ w.r.t. the vertex $w\in X$ is defined by considering the graph structure $\mathcal T_w=(X,F)$ where $(u,v)\in F$ if and only if $\{u,v,w\}\in T$. We can compute the descendant of the two-graph $\mathcal{X}^{\pm}_{2m}$.

\begin{prop}\label{descendant}
 The descendant of $\mathcal{X}^{\pm}_{2m}$ is the graph $\Gamma(O^{\mp}(2m,2))=(V,E)$, plus an isolated vertex.
\end{prop}

\begin{proof}
    Take the two-graph $\mathcal{X}^\pm_{2m}=(X,T)$ as in Proposition \ref{t-g}, and fix $d\in X$. Then $\Theta(d)=1$. The descendant of $\mathcal{X}^\pm_{2m}$ will be denoted as $(\mathcal{X}^\pm_{2m})_d=(X,F)$. As in the proof of Proposition \ref{pr:anal2}, the translation $\tau_d$ gives us the bijection between the vertex sets, where $d$ plays the role of the isolated vertex. We need only to prove the adjacency relation. 
    \begin{align*}
        (a,b)\in F &\Longleftrightarrow \{a,b,d\} \in T \\
        &\Longleftrightarrow \langle a,b\rangle+\langle a,d\rangle+\langle b,d\rangle=0 \\
        &\Longleftrightarrow \langle a+d,b+d\rangle=0 \\
        &\Longleftrightarrow (a+d,b+d) \in E.
    \end{align*}
    Then $\tau_d$ induces an isomorphism between the graphs $(\mathcal{X}^\pm_{2m})_d$ and $\Gamma(O^{\mp}(2m,2))\cup\{0\}$.
\end{proof}

\subsection{The main result on the switching equivalence}
The purpose of this section is to prove the switching equivalence of the strongly regular graphs $NO^{\pm}(4m,2)$ and $NO^{\mp}(2m+1,4)$. For this, we will need to identify the vector spaces $\mathbb{F}_{2}^{4m}$ and $\mathbb{F}_{4}^{2m}$. We start with a technical lemmas.

\begin{lem} \label{lm:stars}
    Let $q=2^h$, and $V$ be an $\mathbb{F}_q$-linear space of dimension $2m$. We consider $V$ as a $2^{hm}$ dimensional $\mathbb{F}_2$-linear space at the same time. 
    \begin{enumerate}[(i)]
        \item Any symplectic $\mathbb{F}_q$-bilinear form $\langle .,. \rangle$ determines a symplectic $\mathbb{F}_2$-bilinear form $\langle .,. \rangle^*$ by
        \[\langle u,v \rangle^* = \Tr_{\mathbb{F}_q/\mathbb{F}_2}(\langle u,v \rangle).\]
        \item Similarly, if $\Theta:V\to \mathbb{F}_q$ is a quadratic form over $\mathbb{F}_q$, then
        \[\Theta^*(u)=\Tr_{\mathbb{F}_q/\mathbb{F}_2}(\Theta(u))\]
        is a quadratic form over $\mathbb{F}_2$. 
        \item If $\Theta$ linearizes to $\langle.,.\rangle$, then $\Theta^*$ linearizes to $\langle.,.\rangle^*$. In particular, $\Theta^*$ is non-singular if and only if $\Theta$ is non-singular.
        \item If $\Theta$ is non-singular, then $\Theta$ and $\Theta^*$ have the same type.
        \item Fix the symplectic forms $\langle .,. \rangle$ and $\langle .,. \rangle^* = \Tr_{\mathbb{F}_q/\mathbb{F}_2}(\langle .,. \rangle)$. Then $\Theta \mapsto \Theta^*$ is a $1$-to-$1$ correspondence between the $\mathbb{F}_q$-quadratic forms linearizing to $\langle .,. \rangle$ and the $\mathbb{F}_2$-quadratic forms linearizing to $\langle .,. \rangle^*$. 
    \end{enumerate}
\end{lem}
\begin{proof} 
    (i), (ii) and (iii) are obvious. We show (iv) by counting the solutions $\Theta^*(x)=0$ with $x\in V$. We have
    \[|\Theta^{-1}(a)|=\begin{cases}
        q^{2m-1} \pm q^{m-1}(q-1) & \text{if $a=0$,} \\
        q^{2m-1} \mp q^{m-1} & \text{if $a\neq 0$.} 
    \end{cases}\]
    In a field $\mathbb{F}_q$ of even characteristic there are $\frac{q}{2}-1$ non-zero elements of absolute trace $0$. This implies 
    \begin{align*}
        |\{x\in V \mid \Tr_{\mathbb{F}_q/\mathbb{F}_2}(\Theta(x)) = 0 \}| &= q^{2m-1} \pm q^{m-1}(q-1)+\Big(\frac{q}{2}-1\Big)(q^{2m-1} \mp q^{m-1}) \\
        &=\frac{1}{2}(q^{2m}\pm q^m) \\
        &= 2^{2hm-1}\pm 2^{hm-1}.        
    \end{align*}
    This is precisely the number of solutions of a quadratic equation of the appropriate type, in a space of dimension $2hm$ over $\mathbb{F}_2$. (v) The sets of both kinds of quadratic forms have cardinalities $q^{2m}=2^{2hm}$. As $\Theta\mapsto \Theta^*$ is injective, (v) follows. 
\end{proof}

We are now able to state our main result. 

\begin{thm} \label{thm:two-graph-structure}
    \begin{enumerate}[(i)]
        \item The two-graph of $NO^\pm(2m,2)$ is isomorphic to $\mathcal{X}^\pm_{2m}$. 
        \item The two-graph of $NO^\mp(2m+1,4)$ is isomorphic to $\mathcal{X}^\pm_{4m}$. 
        \item The strongly regular polar graphs $NO^\pm(4m,2)$ and $NO^\mp(2m+1,4)$ are switching equivalent. 
    \end{enumerate}
\end{thm}
\begin{proof}
    (i) We represent $NO^\pm(2m,2)$ and $\mathcal{X}^\pm_{2m}$ as in Proposition \ref{pr:anal1} and Theorem \ref{thm:two-graph-structure} using a quandratic form $\Theta^*$ of type $\pm$, which linearizes to the symplectic form $\langle .,. \rangle^*$ on $\mathbb{F}_2^{2m}$. Then, $NO^\pm(2m,2)$ and $\mathcal{X}^\pm_{2m}$ have the same set of vertices. In $NO^\pm(2m,2)$, two vertices $a$ and $b$ are adjacent if and only if $\langle a,b\rangle^*=0$, Clearly, for all $a,b\in V$, $\langle a,b\rangle^*\in\{0,1\}$. The triple $\{a,b,c\}$ is in the two-graph associated with $NO^{\pm}(2m,2)$ if and only if there is an odd number of adjacencies, that means there is an even number of 1's in $\{\langle a,b\rangle^*,\langle a,c\rangle^*,\langle a,b\rangle^*\}$. Equivalently,
    \[\langle a,b\rangle^* + \langle a,c\rangle^* + \langle a,b\rangle^*=0.\]
    This proves the claim.
    
    (ii) Let $q=4$, $V=\mathbb{F}_4^{2m}$, $\langle .,. \rangle$ and $\langle .,. \rangle^* = \Tr_{\mathbb{F}_q/\mathbb{F}_2}(\langle .,. \rangle)$ symplectic forms over $\mathbb{F}_4$ and $\mathbb{F}_2$, respectively. Let $\Theta$ be a quadratic form of type $\pm1$ over $\mathbb{F}_4$, which linearizes to $\langle .,. \rangle$. The quadratic form $\Theta^*(x)=\Tr_{\mathbb{F}_q/\mathbb{F}_2}(\Theta(x))$ over $\mathbb{F}_2$ has the same type $\pm1$. We represent $NO^\mp(2m+1,4)$ and $\mathcal{X}^\pm_{4m}$ as in Proposition \ref{pr:anal2} and Theorem \ref{thm:two-graph-structure}. Then $NO^\mp(2m+1,4)$ and $\mathcal{X}^\pm_{4m}$ have the same set of vertices
    \[X^\pm=\{a\in V \mid \Theta^*(a)=1\}.\]
    Moreover,
    \begin{align*}
        \text{$a,b$ are adjacent in $NO^\mp(2m+1,4)$} &\Longleftrightarrow \Theta(a+b)+\langle a,b\rangle^2 =0 \\
        &\Longleftrightarrow \Theta(a)+\Theta(b)+\langle a,b \rangle+\langle a,b\rangle^2 =0 \\
        &\Longleftrightarrow \Theta(a)+\Theta(b)+\langle a,b \rangle^* =0. 
    \end{align*}
    Since $\Tr_{\mathbb{F}_4/\mathbb{F}_2}(\Theta(a))=\Theta^*(a)=1$ and $\Tr_{\mathbb{F}_4/\mathbb{F}_2}(\Theta(b))=\Theta^*(b)=1$, we have
    \[\Tr_{\mathbb{F}_4/\mathbb{F}_2}(\Theta(a)+\Theta(b))=0,\]
    that is, $\Theta(a)+\Theta(b)\in \mathbb{F}_2$. This implies 
    \[\Theta(a+b)+\langle a,b\rangle^2 \in \mathbb{F}_2\]
    for all $a,b\in X^\pm$. In particular, the triple $\{a,b,c\}$ has an even number of non-edges if and only if
    \begin{align*} 
        \Theta(a+b)+\langle a,b\rangle^2 + \Theta(a+c)+\langle a,c\rangle^2 + \Theta(b+c)+\langle b,c\rangle^2 =0,
    \end{align*}
    or equivalently,
    \begin{align*}
        \langle a,b\rangle^* + \langle a,c\rangle^* + \langle b,c\rangle^* =0.
    \end{align*}
    This proves that the two-graph of $NO^\mp(2m+1,4)$ and $\mathcal{X}^\pm_{4m}$ consist of the same set of triples. 

    (iii) The claim follows from (i), (ii), and Proposition \ref{pr:switching-equiv-same}. 
\end{proof}

Theorem \ref{thm:two-graph-structure} shows that $NO^{\pm}(4m,2)$ and $NO^{\mp}(2m+1,4)$ have the same associated two-graph, whose descendant is $\Gamma(O^{\mp}(4m,2))$. Our last result will show the parameters of the Seidel switching between $NO^{\pm}(4m,2)$ and $NO^{\mp}(2m+1,4)$.

\begin{prop}
    The switching sets of the Seidel switching between $NO^{\pm}(4m,2)$ and $NO^{\mp}(2m+1,4)$ are
    \begin{align*}
        A &= \{a\in \mathbb{F}_4^{2m} \mid \Theta(a)=\lambda\}, \\
        B &= \{a\in \mathbb{F}_4^{2m} \mid \Theta(a)=\lambda+1\},
    \end{align*}
    where $\lambda \in \mathbb{F}_4\setminus \mathbb{F}_2$ is a root of $X^2+X+1=0$. The switching sets have size $2^{4m-2}\mp 2^{2m-2}$ and are regular of degree $2^{4m-3}\mp 2^{2m-2}-1$.
\end{prop} 
\begin{proof}
    Let $\Theta$ and $\Theta^*$ be defined as in Theorem \ref{thm:two-graph-structure}. The vertices $a,b$ have the same adjacencies in both graphs if and only if $\Theta(a)=\Theta(b)$. $\Theta^*(a)=\Theta^*(b)=1$ implies $\Theta(a),\Theta(b)\in \{\lambda,\lambda+1\}$. Therefore, the switching sets are the sets $A$, $B$ as given.    Both graphs have $v=2^{4m-1}\mp 2^{2m-1}$ vertices. The degrees and the numbers of common neighbors of non-adjacent vertices are $k=2^{4m-2}-1$, $\mu=2^{4m-3}\pm 2^{2m-2}$, and $k'=2^{4m-2}\mp 3\cdot 2^{2m-2}-1$, $\mu'=2^{4m-3}\mp 2^{2m-1}$, respectively. By \cite[Proposition 1.1.2]{brvan}, the switching sets have size $v/2$ and are regular of degree $d=k-\mu=k'-\mu'$.
\end{proof}

\section{Acknowledgement}
We thank Andries E. Brouwer and Gábor Korchmáros for stimulating discussion and valuable remarks. The research of the first author was supported by project TKP2021-NVA-09, implemented with the support provided by the Ministry of Innovation and Technology of Hungary from the National Research, Development and Innovation Fund, financed under the TKP2021-NVA funding scheme. Partially supported by the NKFIH-OTKA Grant SNN 132625. The research of the last author was partially supported by the Italian National Group for Algebraic and Geometric Structures and their Applications (GNSAGA - INdAM) and by the INdAM - GNSAGA Project \emph{Tensors over finite fields and their applications}, number E53C23001670001.

\appendix

\section{Symplectic groups in even characteristic}\label{sec:appendix}

The theorems of Section \ref{sec:prelim} are well known, but it is hard to find a concise proof for them in the literature. The purpose of the Appendix is to give the most important steps of the proofs of Theorem \ref{th:structure_sp_2m_q}, \ref{th:action_sp_2m_2} and \ref{th:symple-complements}. We use the terminology and notation of Section \ref{sec:prelim}. 

\begin{lem} \label{lm:qforms-equiv}
    Let $\Theta_1,\Theta_2\in\Omega$ be non-degenerate quadratic forms linearizing to the same symplectic form $\langle .,. \rangle$. Then $\Theta_1,\Theta_2\in\Omega$ are $GL(2m,q)$-equivalent if and only if they are $Sp(2m,q)$-equivalent.
\end{lem}
\begin{proof}
    Since $Sp(2m,q)\leq GL(2m,q)$, we only need to prove the right implication. Let $A\in GL(2m,q)$. If $\Theta_1^A=\Theta_2$, then for all $u,v\in V$, 
    \begin{align*}
        \langle u,v\rangle & =\Theta_1(u+v)+\Theta_1(u)+\Theta_1(v) \\
        &=\Theta_2((u+v)A)+\Theta_2(uA)+\Theta_2(vA) \\
        & =\langle uA,vA\rangle.
    \end{align*}
    Hence $A\in Sp(2m,q)$.
\end{proof}

\begin{lem}
The orders of the orthogonal groups are given by the formulas
\begin{align*}
    |O(2m+1,q)| &= q^{m^2}\prod_{i=1}^m(q^{2i}-1), \\
    |O^\pm(2m,q)| &= 2 q^{m(m-1)} (q^m\mp 1)\prod_{i=1}^{m-1}(q^{2i}-1).\\
\end{align*}
\end{lem}
\begin{proof}
    This follows from Theorems 3.12, 14.2 and 14.48 of \cite{grove}. 
\end{proof}

Witt's Extension Theorem \cite[Theorem 5.2]{grove} implies the following.
\begin{lem} \label{lm:orth_orbit}
    Let $\Theta$ be a quadratic form of type $\varepsilon$ on $V=\mathbb{F}_q^n$. Let $a,b\in V$ be two non-zero vectors. There is an element $A\in O^\varepsilon(n,q)$ if and only if $\Theta(a)=\Theta(b)$. \qed
\end{lem}

For a finite field of characteristic two, we have the group isomorphism $O(2m+1,q)\cong Sp(2m,q)$. If $n=2m$ is even, then the alternating form associated to a non-singular quadratic form is non-singular. This embeds $O^\varepsilon(2m,q)$ into $Sp(2m,q)$. In fact, in even dimension, the orthogonal groups are maximal subgroups of the symplectic group. Their index is
\[|O(2m+1,q):O^\pm(2m,q)| = \frac{1}{2} q^m \frac{q^{2m}-1}{q^m\mp 1} = \frac{1}{2}q^m(q^m\pm 1).\]

\subsection{Symplectic transvections}
We introduce now the subset of symplectic transvections, by defining
\begin{equation*}
  \begin{array}{lccc}
  T_a: & V & \longrightarrow & V \\
 & u & \mapsto & u+\langle u,a\rangle a,\\
\end{array}
\end{equation*}
see \cite{grove} for more information about transvections. The following lemma holds:
\begin{lem} \label{lm:transvections}
    \begin{enumerate}[(i)]
        \item $T_a^2=1$, $\forall a\in V$;
        \item $u$ is fixed by $T_a$ if and only if $u\in a^{\perp}$;
        \item $A^{-1}T_a A=T_{aA}$, for any row vector $a\in V$ and symplectic matrix $A\in Sp(2m,q)$. 
        % In fact, \[u(A^{-1} T_a  A) = (uA^{-1} + \langle uA^{-1},a\rangle a) A = u + \langle uA^{-1},a\rangle aA = u + \langle u,aA\rangle aA  = u T_{aA}.\] \qed
    \end{enumerate}
\end{lem}

\begin{lem} \label{lm:traction-formulas}
Let $a,b,c\in V$,  $\gamma\in \mathbb{F}^*_{q}$. The following hold:
    \begin{align}
        \vartheta_a^{T_c}(u)&=\vartheta_a(u)+\langle c,u\rangle^2(\vartheta_a(c)+1), \label{eq:theta1}\\
        \vartheta_b(u)&=\vartheta_a(u)+\langle a+b,u\rangle^2, \label{eq:theta2}\\
        \vartheta_a(a+b)&=\vartheta_0(a)+\vartheta_0(b)+\langle a,b\rangle^2+\langle a,b\rangle, \label{eq:theta3}\\
        \vartheta_a^{T_{\gamma(a+b)}}(u)&=\vartheta_b(u)+(1+\gamma^2+\gamma^4\vartheta_a(a+b))\langle a+b,u\rangle^2, \label{eq:theta4}\\
        1+\gamma^2+\gamma^4\vartheta_a(a+b)&=\gamma^4(\frac{1}{\gamma^4}+\langle a,b\rangle^2+\frac{1}{\gamma^{2}}+\langle a,b\rangle+\vartheta_0(a)+\vartheta_0(b)). \label{eq:theta5}
    \end{align}
\end{lem}
\begin{proof}
\begin{align*}
& \text{For \eqref{eq:theta1}:} & \vartheta_a^{T_c}(u)&=\vartheta_a(uT_c) \\
&&&= \vartheta_a(u+\langle c,u\rangle c) \\ 
&&&= \vartheta_a(u)+\langle u,c\rangle^2\vartheta_a(c)+\langle u,\langle u,c\rangle c\rangle \\
&&&= \vartheta_a(u)+\langle c,u\rangle^2(\vartheta_a(c)+1). \\
%\end{align*}
%\begin{align*}
& \text{For \eqref{eq:theta2}:} & \vartheta_b(u)&=\vartheta_0(u)+\langle b,u\rangle^2 \\
&&&= \vartheta_0(u)+\langle a,u\rangle^2+\langle a+b,u\rangle^2 \\
&&&= \vartheta_a(u)+\langle a+b,u\rangle^2. \\
%\end{align*}
%\begin{align*}
& \text{For \eqref{eq:theta3}:} & \vartheta_a(a+b)&=\vartheta_a(a)+\vartheta_a(b)+\langle a,b\rangle \\
&&& =\vartheta_0(a)+\vartheta_0(b)+\langle a,b\rangle^2+\langle a,b\rangle.
\end{align*}
\eqref{eq:theta4} follows from \eqref{eq:theta1} and \eqref{eq:theta2}:
\begin{align*}
\vartheta_a^{T_{\gamma(a+b)}}(u)&=\vartheta_a(u)+(1+\vartheta_a(\gamma(a+b))\langle u,\gamma(a+b)\rangle^2 \\
&=\vartheta_b(u)+\langle a+b,u\rangle^2+(1+\vartheta_a(\gamma(a+b))\gamma^2\langle a+b,u\rangle^2 \\
&=\vartheta_b(u)+(1+\gamma^2+\gamma^4\vartheta_a(a+b))\langle a+b,u\rangle^2.
\end{align*}
Finally, we obtain \eqref{eq:theta5} by \eqref{eq:theta2},
\begin{align*}
1+\gamma^2+\gamma^4\vartheta_a(a+b)&=1+\gamma^2+\gamma^4(\vartheta_0(a)+\vartheta_0(b)+\langle a,b\rangle+\langle a,b\rangle^2) \\
&=\gamma^4(\frac{1}{\gamma^4}+\frac{1}{\gamma^{2}}+\langle a,b\rangle^2+\langle a,b\rangle+\vartheta_0(a)+\vartheta_0(b)). \qedhere
\end{align*}
\end{proof}

\begin{lem} \label{lm:traction-omega}
 Let $a,b\in V$, $a\neq b$. The following are equivalent:
 \begin{enumerate}[(i)]
     \item $\exists c\in V:\vartheta_a^{T_c}=\vartheta_b$;
     \item $\exists\gamma\in\mathbb{F}^*_q:\vartheta_a^{T_{\gamma(a+b)}}=\vartheta_b$;
     \item $\Tr_{\mathbb{F}_q/\mathbb{F}_{2}}(\vartheta_0(a))=\Tr_{\mathbb{F}_q/\mathbb{F}_{2}}(\vartheta_0(b))$;
    \item the polynomial $t^2+t+\vartheta_0(a)+\vartheta_0(b)$ is reducible over $\mathbb{F}_q$.
 \end{enumerate}
\end{lem}

\begin{proof}
 (ii) implies (i) trivially. It is well known that $t^2+t+\lambda$ is reducible over $\mathbb{F}_q$ if and only if $\Tr_{\mathbb{F}_q/\mathbb{F}_{2}}(\lambda)=0$, see for instance \cite[Section 1.4]{Hirschfeld1} hence (iii)$\Leftrightarrow$(iv). Assume (iv) and let $t\in \mathbb{F}_q$ such that $t^2+t+\vartheta_0(a)+\vartheta_0(b)=0$. If $\gamma = \frac{1}{\sqrt{\langle a,b\rangle+t}}$, then \eqref{eq:theta4} implies (ii). Conversely, if (ii) is true, then $1+\gamma^2+\gamma^4\vartheta_a(a+b)=0$, and $t=\frac{1}{\gamma^2}+\langle a,b \rangle$ is a root of $t^2+t+\vartheta_0(a)+\vartheta_0(b)$ in $\mathbb{F}_q$. This proves $(ii)\Leftrightarrow (iv)$. In order to show $(i)\Rightarrow (ii)$, let 
 \begin{align*}
 \vartheta_a^{T_c}(u) &= \vartheta_a(u)+\langle u,c\rangle^2(1+\vartheta_a(c)) \\
 &= \vartheta_0(u)+\langle u,a\rangle^2+\langle u,c\sqrt{\vartheta_a(c)+1}\rangle^2 \\
 &=\vartheta_{a+c\sqrt{\vartheta_a(c)+1}}(u).
\end{align*}
By (i), $b=a+c\sqrt{\vartheta_a(c)+1}$, and since $a\neq b$, we have $\vartheta_a(c)+1\neq0$ and $c=\frac{a+b}{\sqrt{\vartheta_a(c)+1}}$.
\end{proof}

\subsection{Proofs of the theorems on the structure of symplectic groups}

\begin{proof}[Proof of Theorem \ref{th:structure_sp_2m_q}]
(i) Let $\Omega^+$ and $\Omega^-$ denote the set of hyperbolic and elliptic quadratic forms in $\Omega$. The elements in $\Omega^\pm$ are $GL(2m,q)$-equivalent, hence by Proposition \ref{lm:qforms-equiv}, they are $Sp(2m,q)$-equivalent. This means that $Sp(2m,q)$ has these two orbits in $\Omega$. As $\Tr_{\mathbb{F}_q/\mathbb{F}_2}(\vartheta_0(a))$ can take only two values $0$ and $1$, Lemma \ref{lm:traction-omega} implies that the orbits are also characterized by the equations
\begin{align*}
    \Tr_{\mathbb{F}_q/\mathbb{F}_2}(\vartheta_0(a))&=0 \\
    \Tr_{\mathbb{F}_q/\mathbb{F}_2}(\vartheta_0(a))&=1.
\end{align*}
Clearly, $\vartheta_0\in \Omega^+$ and $\Tr_{\mathbb{F}_q/\mathbb{F}_2}(\vartheta_0(0))=0$, which proves the claim.

(ii) If $\beta \in \mathbb{F}_q\setminus \{0\}$, then the number of solutions $a\in \mathbb{F}_q^{2m}$ of equation $\vartheta_0(a)=\beta$ is $q^{2m-1}-q^{m-1}$. For $\beta=0$, the number of solutions is $q^{2m-1}+q^{m-1}(q-1)$. This allows us to compute the number of solutions of $\Tr_{\mathbb{F}_q/\mathbb{F}_2}(\vartheta_0(a))=0$ and $\Tr_{\mathbb{F}_q/\mathbb{F}_2}(\vartheta_0(a))=1$. 

(iii) For any $a,b\in \mathbb{F}_q^{2m}$, we have
\[\vartheta_{a+b}=\vartheta_0(u)+\langle a+b,u \rangle^2 = \vartheta_a(u) + \langle b,u \rangle^2.\]
Assume $\vartheta_a^A=\vartheta_a$. Then 
\begin{align*}
    \vartheta_{a+b}^A(u) &= \vartheta_{a+b}(uA^{-1}) \\
    &= \vartheta_a(uA^{-1}) + \langle b,uA^{-1} \rangle^2 \\
    &= \vartheta_a(u) + \langle bA,u \rangle^2 \\
    &= \vartheta_{a+bA}(u).    
\end{align*}
This means that $b\mapsto \vartheta_{a+b}$ induces an equivalence between the linear action of the symplectic group on $\mathbb{F}_q^{2m}$, and the action of $Sp(2m,q)$ on $\Omega$. 
\end{proof}

\begin{proof}[Proof of Theorem \ref{th:action_sp_2m_2}]
The claim is equivalent to the fact that the stabilizers $R^\pm$ have three orbits: $\{\vartheta_a\}$, $\Omega^\pm \setminus\{\vartheta_a\}$, and $\Omega^\mp$. By Theorem \ref{th:structure_sp_2m_q}(iii), the stabilizers are permutation equivalent to the linear action of the orthogonal group $O^\pm(2m,2)$. Indeed, for $q=2$, the orthogonal groups have three orbits $\{0\}$, $\{ u \in \mathbb{F}_2^{2m}\setminus\{0\} \mid \vartheta_a(u)=0\}$ and $\{ u \in \mathbb{F}_2^{2m} \mid \vartheta_a(u)=1\}$. 
\end{proof}

\begin{proof}[Proof of Theorem \ref{th:symple-complements}]
The symplectic group $H_1=Sp(2m,q)$ itself is a complement to $N$ in $ASp(2m,q)$, the orbits of $H_1$ have length $1$ and $2^{2m}-1$. To find the second complement, we identify $\mathbb{F}_q^{2m}$ with $\Omega$ by $a\mapsto \vartheta_a$. The action of $Sp(2m,q)$ induces an action on $\mathbb{F}_q^{2m}$ which we denote by $a\mapsto a^A$ as well. 
Straightforward computation shows
\begin{align*}
\vartheta_{a+b+c}(u) &=\vartheta_0(u)+\langle a+b+c,u \rangle^2 \\
&=\vartheta_0(u)+\langle a,u \rangle^2+\vartheta_0(u)+\langle b,u \rangle^2+\vartheta_0(u)+\langle c,u \rangle^2,\\
&=\vartheta_a(u)+\vartheta_b(u)+\vartheta_c(u),
\end{align*}
and
\begin{align*}
\vartheta_a^A(u)&=\vartheta_a(uA^{-1}) \\
&=\vartheta_0(uA^{-1}) +\langle a,uA^{-1}\rangle^2\\
&=\vartheta_0^A(u)+\vartheta_0(u)+\vartheta_0(u)+\langle aA,u\rangle^2\\
&=\vartheta_0^A(u)+\vartheta_0(u)+\vartheta_{aA}(u).
\end{align*}
Fix $A\in Sp(2m,q)$ and define $b$ as $\vartheta_0^A=\vartheta_b$. Then, 
\[\vartheta_a^A = \vartheta_{aA+b}.\]
In other words, $a\mapsto a^A=aA+b$ is an element of $ASp(2m,q)$. Obviously, the maps $\{a\mapsto a^A \mid A\in Sp(2m,q)\}$ form a group $H_2\cong Sp(2m,q)$. The only translation in $H_2$ is the trivial one; therefore, $H_2$ is a complement of $N$. The orbits of $H_2$ have length $\frac{1}{2}(q^{2m}+q^m)$, $\frac{1}{2}(q^{2m}-q^m)$. The identification of $\mathbb{F}_q^{2m}$ and $\Omega$ is an equivalence between the actions of $H_2$ on the orbits and the actions of $Sp(2m,q)$ on $\Omega^+$ and $\Omega^-$. 
\end{proof}


\begin{thebibliography}{999}
 \bibitem{Bose}R. C. Bose, \textit{Strongly regular graphs, partial geometries and partially balanced designs}, Pacific Journal of Mathematics, 1963, 13(2), pp. 389-419. MR0157909
 \bibitem{BH}A. E. Brouwer, W. H. Haemers, \textit{Spectra of Graphs}, Springer, New York, 2012. MR2882891
 \bibitem{brVL}A. E. Brouwer, J. H. Van Lint, \textit{Strongly regular graphs and partial geometries}, in: "Enumeration and Design", D. H. Jackson, S. A. Vanstone (Eds.), pp. 85-122, Academic Press, Toronto, 1984. MR0782310
  \bibitem{brvan}A. E. Brouwer, H. Van Maldeghem, \textit{Strongly Regular Graphs}, Encyclopedia of Mathematics and its Applications, Cambridge University Press, 2022. MR4350112
 \bibitem{table}F. C. Bussemaker, R. A. Mathon, J. J. Seidel, \textit{Tables of two-graphs} in: "Combinatorics and Graph Theory", S. B. Rao (Ed.), Proceedings of the Symposium Held at the Indian Statistical Institute, Calcutta 1980, pp. 70–112, Lecture Notes in Mathematics 885, Springer-Verlag, Berlin, 1981. MR0655610
\bibitem{grove}L. C. Grove, \textit{Classical Groups and Geometric Algebra}, American Mathematical Society, 2002. MR1859189
 \bibitem{Hirschfeld1}J. W. P. Hirschfeld, \textit{Projective Geometries over Finite Fields}, Oxford Mathematical Monographs, 1979. MR0554919
\bibitem{Hirschfeld3}J. W. P. Hirschfeld, J. Thas, \textit{General Galois Geometries}, Springer-Verlag London, 2016. MR3445888
  \bibitem{seidel}J. J. Seidel \textit{Graphs and two-graphs}. In: Proceedings of the Fifth Southeastern Conference on Combinatorics, Graph Theory and Computing (Florida Atlantic University, Boca Raton, Fla., 1974). Utilitas Mathematica, Winnipeg, Man.; 1974. pp. 125–143. Congressus Numerantium, No. X. MR0364028
\bibitem{Tay}D. E. Taylor, \textit{Regular 2-graphs}, Proceedings of the London Mathematical Society, 1977, 35(3), pp. 257-274. MR0476587
\end{thebibliography}
\end{document}